\documentclass[12pt, reqno]{amsart}

\usepackage[english]{babel}
\usepackage[utf8x]{inputenc}
\usepackage[T1]{fontenc}
\usepackage[all,cmtip]{xy}

\usepackage[alphabetic, initials, lite]{amsrefs} 
\usepackage{fullpage}
\usepackage{setspace}
\usepackage{amssymb,amsmath,amsthm,amscd,epsfig}
\usepackage{graphicx}
\usepackage[colorinlistoftodos]{todonotes}
\usepackage[colorlinks=true, allcolors=blue]{hyperref}

\newtheorem{thm}{Theorem}[section]
\newtheorem{lem}[thm]{Lemma}
\newtheorem{cor}[thm]{Corollary}
\newtheorem{prop}[thm]{Proposition}

\theoremstyle{definition}

\newtheorem{question}[thm]{Question}

\newtheorem{remark}[thm]{Remark}
\newtheorem{case}{Case}

\theoremstyle{remark}

\newtheorem{subcase}{Subcase}

\numberwithin{equation}{section}

\def\Z{\ifmmode{{\mathbb Z}}\else{${\mathbb Z}$}\fi}
\def\Q{\ifmmode{{\mathbb Q}}\else{${\mathbb Q}$}\fi}
\def\C{\ifmmode{{\mathbb C}}\else{${\mathbb C}$}\fi}
\def\P{\ifmmode{{\mathbb P}}\else{${\mathbb P}$}\fi}
\def\H{\ifmmode{{\mathrm H}}\else{${\mathrm H}$}\fi}
\def\G{\ifmmode{{\mathbb G}}\else{${\mathbb G}$}\fi}
\def\R{\ifmmode{{\mathbb R}}\else{${\mathbb R}$}\fi}
\def\F{\ifmmode{{\mathbb F}}\else{${\mathbb F}$}\fi}
\def\O{\ifmmode{{\cal O}}\else{${\cal O}$}\fi}
\def\D{\ifmmode{{\cal{D}}^b}\else{${{\cal{D}}^b}$}\fi}

\newcommand{\Qbar}{{\overline{\Q}}}
\newcommand{\Zhat}{{\hat{\Z}}}

\newcommand{\kbar}{{\overline{k}}}

\newcommand{\kk}{{\mathbf k}}
\newcommand{\defi}[1]{\textsf{\textbf{#1}}}

\newcommand{\calA}{{\mathcal A}}

\newcommand{\calE}{{\mathcal E}}

\newcommand{\calH}{{\mathcal H}}

\newcommand{\calO}{{\mathcal O}}

\newcommand{\calR}{{\mathcal R}}

\newcommand{\calU}{{\mathcal U}}

\newcommand{\calZ}{{\mathcal Z}}

\DeclareMathOperator{\Frob}{Frob}

\DeclareMathOperator{\inv}{inv}

\DeclareMathOperator{\Gal}{Gal}

\DeclareMathOperator{\Sym}{Sym}

\DeclareMathOperator{\ev}{ev}
\DeclareMathOperator{\PGL}{PGL}

\DeclareMathOperator{\tors}{tors}
\DeclareMathOperator{\et}{et}

\DeclareMathOperator{\red}{red}

\DeclareMathOperator{\SL}{SL}
\DeclareMathOperator{\PSL}{PSL}
\DeclareMathOperator{\GL}{GL}

\usepackage{todonotes}

\newcommand{\A}{\mathbb{A}}

\newcommand{\injects}{\hookrightarrow}
\newcommand{\isom}{\simeq}

\newcommand{\Br}{\mathrm{Br}}

\usepackage[shortlabels]{enumitem}
\usepackage{tikz-cd}

\title{Rational points on conic bundles over elliptic curves}

\author{Jennifer Berg}
\address{Department of Mathematics, Bucknell University, Lewisburg, PA
17837, USA}
\email{jsb047@bucknell.edu}
\urladdr{http://sites.google.com/view/jenberg}

\author{Masahiro Nakahara}
\address{Department of Mathematical Sciences, University of Bath, Bath, BA2 7AY, UK}
\email{mn634@bath.ac.uk}
\urladdr{https://sites.google.com/view/masahiro-nakahara/home}

\begin{document}

\begin{abstract}
We study rational points on conic bundles over elliptic curves with positive rank over a number field. We show that the \'etale Brauer--Manin obstruction is insufficient to explain failures of the Hasse principle for such varieties. We then further consider properties of the distribution of the set of rational points with respect to its image in the rational points of the elliptic curve. In the process, we prove results on a local-to-global principle for torsion points on elliptic curves over $\Q$.
\end{abstract}

\maketitle

\section{Introduction}

The purpose of this paper is to examine the arithmetic of varieties fibered over elliptic curves of positive Mordell--Weil rank over a number field $k$. The Brauer--Manin and \'etale--Brauer obstructions to rational points on quadric bundles were shown to be insufficient by \cite{ctps}*{\S5} for those varieties fibered over an elliptic curve with a single rational point. We show that the pathologies of these obstructions persist, even for conic bundle surfaces $X$ fibered over positive rank elliptic curves $E$. When $X(k)$ is nonempty, however, the distribution of the $k$-rational points remains largely unexplored. In this context, it is natural to consider properties of the image of $X(k)$ inside of the set of rational points of $E$. We prove that there exist families of conic bundles over elliptic curves defined over $\Q$ such that the image of $X(\Q)$ does not contain a translate of any finite index subgroup in $E(\Q)$. Along the way, we prove a local-to-global principle for torsion on elliptic curves over $\Q$. 

\subsection*{Insufficiency of obstructions}
Over the last two decades, several examples of varieties $X$ over a number field $k$ were constructed without $k$-rational points, but with non-empty obstruction sets $X(\A)^{\Br X}$ and $X(\A)^{\et,\Br}$ ~\cites{Sko99, Poonen10, HaSko14}. In 2016, Colliot-Th\'el\`ene, P\'al, and Skorobogatov \cite{ctps} unified many of the previously known results which each used the trick of a fibration over a curve with a single rational point. They provided methods for constructing new examples of $k$-varieties $X$ fibered over a curve $C$ with a single $k$-rational point, such that $X(k) = \emptyset$, but $X(\A_k)^{\et, \Br} \ne \emptyset$. (For a detailed description of the obstruction sets $X(\A)^{\Br},\, X(\A)^{\et, \Br}$ and the Brauer-Manin pairing, see \cite{Poonen10}.) In particular, their work demonstrated the disparity between the arithmetic of  geometrically rational surfaces, for which the Brauer-Manin obstruction is conjectured to explain all failures of the Hasse principle \cite{CTS80}, and that of conic bundle surfaces over curves of genus at least 1. We prove that the aforementioned obstructions to the Hasse principle are insufficient even when the base curve is allowed to have infinitely many $k$-rational points.


\begin{thm}\label{thm:main1}
There exists an elliptic curve $E$ over a number field $k$ with positive rank over $k$ and a conic bundle $X\to E$ such that $X(\A)^{\et,\Br}\neq\emptyset$ but $X(k)=\emptyset$.
\end{thm}

\subsection*{Distribution of rational points on conic bundles over elliptic curves} The primary focus of this paper is devoted to studying the distribution of $\Q$-rational points on conic bundle surfaces $\phi \colon X \to E$. In \S\ref{sec:distribution} we focus on the case of a Ch\^atelet conic bundle over an elliptic curve which is a smooth projective model over $E$ of the conic
$$u^2-av^2=f(x)w^2$$
over the function field of $E$ (see \S\ref{subsec: ChateletSurfE} for the precise definition). We investigate whether $\phi(X(\Q))$, considered as a subset of $E(\Q)$, contains a translate of a finite index subgroup. To that end, we define a set $\Omega$ of primes, depending on $\phi \colon X\to E$, and show that if for any integer $n$, there are infinitely many primes $p\in\Omega$ for which
\begin{equation}\label{eqn:gcdintro}
\gcd(n,|E(\F_p)|)=1,
\end{equation}
then $\phi(X(\Q))$ contains a finite index subgroup of $E(\Q)$. In \S\ref{sec:torsionpoints}, we give a sufficient condition for \eqref{eqn:gcdintro} to hold, thereby giving the following result:

\begin{thm}\label{thm:main2}
Let $E$ be an elliptic curve over $\Q$ with $E(\Q)\isom\Z^r$ for some $r>0$, satisfying some mild conditions on the Galois representation for torsion points. Then for any Ch\^atelet conic bundle $X$ over $E$ with a singular fiber over a point $P {\neq \calO}\in E(\Q)$, the image of $X(\Q)$ in $E(\Q)$ does not contain a translate of a finite index subgroup inside $E(\Q)$.
\end{thm}

See Theorem \ref{thm:arithprog} for the precise conditions on the Galois representations of $E$. In particular, any elliptic curve with no exceptional primes satisfies these conditions. Furthermore, we prove that when $E$ has full 2-torsion subgroup defined over $\Q$ (hence \eqref{eqn:gcdintro} fails miserably), it is possible for $\phi(X(\Q))$ to contain a translate of a finite index subgroup of $E(\Q)$.

\subsection*{Local-to-global properties for reduction of elliptic curves} The last part of the paper is dedicated to proving that under mild hypotheses on $E$, \eqref{eqn:gcdintro} holds for infinitely many primes $p\in\Omega$. A necessary condition is that $E'(\Q)[n]=0$ for any isogenous elliptic curve $E'$ and integer $n$. One can ask whether this condition is sufficient, i.e., whether a local-to-global principle holds for torsion points. Katz proved that this holds for a similar question for the property that $n\!\mid\! \#E(\Q)_\textup{tors}$~\cite{katz}. However, even when $n$ is a prime, this does not imply that \eqref{eqn:gcdintro} holds for infinitely many primes $p\in\Omega$ as we discuss in \S\ref{sec:torsionpoints}. For other results concerning local-to-global principles for elliptic curves, see~\cite{Vogt18}, \cite{sut12}.

Let $\zeta_3$ denote a primitive $3$rd root of unity. Let $\rho_\ell\colon \Gal(\Qbar/\Q)\to \GL_2(\Z/\ell\Z)$ be the Galois representation associated to $\ell$-torsion points. The main result of this section is as follows:

\begin{thm}\label{thm:main3}
Let $K/\Q$ be a nontrivial abelian extension not contained in $\Q(E[2])$ or $\Q(\zeta_3)$. Let $n>1$ be an integer and suppose that $\rho_\ell$ is surjective for all except possibly one odd prime $\ell$ dividing $n$. If $\rho_\ell$ is not surjective and $2\ell \mid n$, assume further that $\Q(E[2])\cap~\Q(E[\ell])=\Q$. Then the following are equivalent:
\begin{enumerate}
    \item $E(\F_p)[n]\neq0$ for all but finitely many primes $p$ that do not split completely in $K$.
    \item There exists a curve $E'$ $\Q$-isogenous to $E$ such that $E'(\Q)[n]\neq0$.
\end{enumerate}
\end{thm}

We may take $K=\Q(\sqrt{a})$ in our application to studying the image of rational points. We briefly discuss the necessity of these conditions on $E$ in Theorem \ref{thm:main3} in the end.

This paper is organized as follows. In \S\ref{conic}, we first prove that the Brauer-Manin obstruction is insufficient to explain all failures of the Hasse principle for conic bundles over many varieties whose rational points fail to be dense in the product of the real topologies. We then prove Theorem \ref{thm:main1}. In \S\ref{sec:distribution}, we turn our attention to the distribution of rational points on conic bundles over elliptic curves. We prove Theorem \ref{thm:main2} using Theorem \ref{thm:main3}. In \S\ref{sec:torsionpoints}, we study properties of torsion points on elliptic curves and prove Theorem \ref{thm:main3}.

\subsection*{Acknowledgements} We thank J.-L. Colliot-Th\'el\`ene, Dan Loughran and Tony V\'arilly-Alvarado for fruitful discussions and their comments on a preliminary draft of this paper. We thank the anonymous reviewer for their suggestions on both mathematical content and exposition. Part of the work was done at the Institut Henri Poincar\'e, we thank the IHP for their hospitality.

\section{Insufficiency of the Brauer--Manin obstructions}\label{conic}

For definition and overview of the Brauer--Manin obstruction, see \cite{sko96}*{\S5.2}. A recurring technique in showing failures of the Brauer--Manin obstruction is to use the real places of the number field $k$ (e.g., \cite{ctps}). More precisely, given an adelic point $\{P_v\}\in X(\A)$, we can modify the point $P_\infty$ by some other point $Q_\infty$ on the same real component of $X_\R$, without changing the Brauer--Manin pairing. One can then construct a fibration $Y\to X$ such that for each $v$, the fiber above $P_v$ has $k_v$-points, but no rational point on $X$ can approximate $\{P_v\}$. Conic bundles are often used in this setting since their Brauer groups are well understood and can often be expressed in terms of the Brauer group of the base.

Recall from \cite{sko96}*{Definition 0.1} that a scheme of finite type over a perfect field $k$ is called \defi{split} if it contains a geometrically irreducible component of multiplicity one.

\subsection{Example over a general base using real places}
\begin{thm}\label{BMinsuff} Let $X$ be a geometrically irreducible projective variety of dimension $\geq2$ over a number field $k$ such that $X(k)\neq\emptyset$. Assume further that $X(\R)$ is connected for all real embeddings of $k$. 
Let $d=[k:\Q]$ and write
$d=r+2s$,
where $r$ is the number of real embeddings and $s$ is the number of complex embeddings.
\begin{enumerate}
\item If $r$ is even,  suppose that $X(k)\to\prod_{v\text{ real}} X(k_v)$ is not dense. 
\item  If $r$ is odd, suppose that there is a real place $v_0$ such that $X(k)\to \prod_ { \substack{v\text{ real} \\  v\neq v_0}}  X(k_v)$ is not dense.

\end{enumerate}
Then there is a conic bundle $Y\to X$ such that $Y(\A)^\Br\neq\emptyset$ but $Y(k)=\emptyset$.
\end{thm}

\begin{proof}
If $r$ is even, let $U=\prod_{v\text{ real}} U_v$ where each $U_v\subseteq X(k_v)$ is a nonempty analytic open set, such that $U$ does not intersect $X(k)$.
If $r$ is odd, let $U=\prod_{\substack{v\text{ real} \\ v\neq v_0}} U_v$ where each $U_v\subseteq X(k_v)$ is a nonempty open set, such that $U$ does not intersect $X(k)$.

Consider the conic
\[ z_0^2+az_1^2=bz_2^2 \]
with $a,b \in k$ and $a$ is totally positive (positive under all real embeddings of $k$).  If $r$ is even, assume that this conic has local points everywhere except at the real places. If $r$ is odd, then suppose this conic has local points everywhere except at real places other than $v_0$. We now build a conic bundle $Y$ over $X$ such that one of the fibers will be isomorphic to this conic.

Let $\phi \colon X \to \P^n$ be an embedding into projective space.  Let $(u_v)\in U$. Pick projective coordinates $x_0,
\ldots,x_n$ for $\P^n$ such that the affine patch $\calU\subset \P^n$ defined by $x_n\neq0$ contains $\phi(u_v)$ for each real place $v$ and a rational point in $\phi(X(k))$. There exists an open set $W_v\subset \P^n(k_v)$ containing $\phi(u_v)$, such that $W_v\cap \phi(X(k_v))\subseteq \phi(U_v)$. Let $y=[y_0,\ldots,y_n]\in \P^n(k)$ be a point which approximates $u_v$ for each $v$ arbitrarily closely. Define a section $g\in \calO_{\P^n}(2)$ whose restriction to $\calU$ is given by
$$g|_\calU=-\sum_{i=0}^{n-1}(x_i/x_n-y_i/y_n)^2+\epsilon,$$
where $\epsilon$ is a small enough positive rational number such that $g$ is negative on $\P^n(k_v)\setminus W_v$ but $g(u_v)$ is positive. Note that the hypersurface defined by $g=0$ is smooth. Also assume there exists a point $z\in X(k)$ such that $\phi(z)\in \calU$ and $g(\phi(z))$ is totally negative. Thus $g$ is negative on $\phi(X)(k_v)\setminus \phi(U_v)$.

Write $g|_\calU=g_1/g_2$ where $g_1,g_2$ are homogeneous quadratic polynomials in the coordinate ring $k[x_0,\ldots,x_n]$. By the Bertini irreducibility theorem, since $\dim X \ge 2$, there is a dense open set of quadric hypersurfaces $H$ in $\P^n$ such that $H\cap X$ is irreducible. Hence we deform the coefficients of $g_1$ slightly, so that $g$ is still negative on $\phi(X)(k_v)\setminus \phi(U_v)$ and positive at $u_v$, while making the intersection of $V:=\{g_1=0\}$ and $X$ irreducible. Define $L=\phi^*\calO(1)$ and $f=\phi^*g\in L^{\otimes2}$. Let $\calE$ be the vector bundle $\calO_X\oplus\calO_X\oplus L$. Define the section
$$s:=1\oplus a\oplus \frac{b}{f|_{\calU}(z)}f\in\Gamma(X,\calO_X\oplus\calO_X\oplus L^{\otimes2})\subset \Gamma(X,\Sym^2\calE).$$
Then the vanishing of $s$ defines a conic bundle $\pi\colon Y\to X$ with a dense open affine subset given by
\[ z_0^2+az_1^2=\frac{b}{f|_{\calU}(z)}f|_\calU. \]
For any real place $v$, $Y$ has $v$-adic points over $u_v$. For any finite place $v$, the fiber above $z$ has $v$-adic points by hypothesis. Thus $Y$ is everywhere locally soluble. Since for each $U_v$, there are no $v$-adic points above $X(k_v)\setminus U_v$, and $X(k)\cap U=\emptyset$, so $Y$ has no rational points.

Let $(y_v) \in Y(\A)$ be an adelic point of $Y$ which maps to $z$ for each finite place $v$. Since $X(\R)$ has only one connected component, the map $\ev_\calA\colon X(k_v)\to \Q/\Z$ is constant on any real place $v$ and $\calA\in\Br(X)$. Hence $\sum_v\inv_v \calA(\pi(y_v))=\sum_v\inv_v\calA(z_v)=0$ where $(z_v)$ is the image of $z\in X(k)$ in $X(\A)$. Thus $\pi((y_v))\in X(\A)^\Br$. The irreducibility of $\phi^*g_1=0$ implies that there exists some codimension one point $P$ on $X$ such that for any other codimension one point $Q\neq P$ on $X$, the fiber $Y_{k(Q)}$ is nonsplit. Hence $\Br(Y) = \pi^{\ast}\Br(X)$ by~\cite{ctps}*{Prop 2.2}, and so $(y_v) \in Y(\A)^{\Br}$. Thus, there is no obstruction arising from $\Br(Y)$.
\end{proof}

\begin{remark} The proof of Theorem \ref{BMinsuff} necessitates the existence of at least 2 real embeddings of the ground field $k$. In particular, it does not work over $k = \Q$. The proof furthermore requires $X$ to have dimension at least 2, but this condition is likely not necessary. However, in order to obtain examples for which the \emph{\'etale-Brauer} obstruction is insufficient to explain failures of the Hasse principle, we restrict our focus to a class of conic bundle surfaces.  
\end{remark}

\subsection{Ch\^atelet conic bundles over elliptic curves} \label{subsec: ChateletSurfE}
Let $E/k$ be an elliptic curve of positive rank over number field $k$ with short Weierstrass equation
$ \label{eqn: Eoverk} y^2=x^3+bx+c, $
and let $\pi\colon E\to \P^1_x$ be the morphism sending $(x,y)$ to $x$. Consider the Ch\^atelet surface $ Y\to \P^1$ defined by
\begin{equation} \label{eqn: conicbundle} u^2-av^2=f(x)w^2,\end{equation} 
where $f(x)$ is a separable polynomial in $k[x]$ and $a\in k^\times\backslash k^{\times2}$ is square free. Suppose that the ramification locus of $Y\to\P^1$ is disjoint from the ramification locus of the double cover $E \to \P^1$. Let $X:=Y\times_{\P^1} E$, which will necessarily be smooth by the requirements on the ramification of the maps involved (\cite{poonen}). We will henceforth refer to such $\phi\colon X\to E$ as a \defi{Ch\^atelet conic bundle over an elliptic curve} given by \eqref{eqn: conicbundle}. The generic fibre $X_{k(E)}$ is the conic defined by the quaternion algebra $\left(a, f(x)\right) \otimes_{\kk(\P^1)} \kk(E)$. If this algebra ramifies at $P\in E(k)$, then the fiber above $P$ has no $k$-points.

The following theorem is motivated by the construction in \cite{ctps}*{\S5}, in which they produced a conic bundle over an elliptic curve with a single rational point such that the \'etale Brauer--Manin obstruction was insufficient. We construct such an example of Ch\^atelet conic bundle over an elliptic curve of positive rank following their techniques.

\begin{thm} \label{Cor: BMO real quad} There exists a real quadratic field $k$, an elliptic curve $E$ with a point $Q\in E$ and a Ch\^atelet conic bundle $\phi \colon X \to E$ satisfying the following properties:
\begin{itemize}
\item the fibers of $\phi \colon X \to E$ are conics;
\item $E(k)=\Z$ and all fibers above $E(k) \setminus Q$ are split;
\item $X(k)=\emptyset$ and $X(\A)^{\et,\Br}\neq\emptyset$.
\end{itemize}
\end{thm}

\begin{proof} 
Let $E$ be the elliptic curve over $\Q$ defined by $y^2 = x^3 - 432x + 15120$. Then $E(\Q) \isom \Z$, and $E(\R)$ is connected. Let $k = \Q(\sqrt{5})$ so that $E(k) \isom E(\Q)$. Write the defining equation of $E$ as $y^2 = r(x)$. Let $a'=3+9\sqrt{5}$. Then $r(a')=17496$ and let $K:=k(\sqrt{17496})$ so that $P:=(3+9\sqrt{5},\sqrt{17496})$ is a point on $E(K)$. Write $P=mQ$ for some $Q\in E(K)$ where $Q$ is nondivisible. Let $\sigma\in \Gal(K/k)$ be the unique nontrivial element. Then $P^\sigma=-P$ implies that $m(Q^\sigma+Q)=0$. A {\tt Magma}~\cite{magma} computation shows that $E(K)$ is torsion free, so $Q^\sigma=-Q$. Hence the $x$-coordinate $a:=x(Q)$ is defined over $k$, but not $\Q$ since $k(Q)$ is a biquadratic extension of $\Q$. Let $\tau\in\Gal(k/\Q)$ be a generator. Choose a rational number $b\in\Q$ such that $b$ lies in the interval  between $a$ and $a^\tau$ and such that $v_5(b)=0$. Then $r(b) \equiv b(b^2-2) \pmod{5}$ whence $v_5(r(b))=0$. Since $K$ is a totally real field, we have $r(a) >0$. Moreover by continuity of $r(x)$ we can choose $b$ sufficiently close to $a$ so that $r(b)>0$ as well. Let $n$ be the product of all primes appearing in $r(b)$. Let $p$ be a prime such that $-p\equiv 1\bmod(8n)$ and $p\equiv 2\bmod 5$. Note that $p$ is inert in $k$. Then the quaternion algebra over $k$ defined by
$$\left(-r(b),-p\right)_2$$
is ramified only at the two real places. Indeed it is unramified at all primes appearing in $2r(b)$ by construction, so the only possibilities for ramification are at $p$ and the two real places. Since it is clearly ramified at the two real places, by Hilbert's reciprocity law, it must be unramified at $p$.

Fix a point $P_0:=(x_0,y_0)\in E(k)$ such that $(x_0-a)(x_0-b)$ is totally positive; this is possible since $E(k)$ is dense in $E(\R)$. Let $\phi \colon X \to E$ be the Ch\^atelet conic bundle over $E$ defined by
\[ u^2  + r(b) v^2 = -p \frac{(x-a)(x - b)}{(x_0 - a)(x_0 - b)} w^2. \]

We first show that $X$ is everywhere locally soluble. Note that the fiber above $P_0$ has local points at all the finite places by the calculation above. For each finite place $v$, let $M_v$ be any $k_v$ point on the fiber above $P_0$. 
Let $v_1$ and $v_2$ denote the two real places, where $k\to k_{v_1}\isom \R$ sends $\sqrt{5}$ to itself while $k\to k_{v_2}\isom \R$ sends $\sqrt{5}$ to $-\sqrt{5}$. By choosing $x \in \R$ with $r(x) > 0$ in the interval between $a$ and $b$ for $v_1$ (resp. in the interval between $a^\tau$ and $b$ for $v_2$) we find that $(x-a)(x-b)<0$ (resp. $(x-a^\tau)(x-b) < 0$). Hence the fiber above $(x,y)$ has a $k_{v_i}$-point $M_{v_i}$ for $i=1,2$. This defines a point $\{M_v\}\in X(\A)$.

To show $X$ has no $k$-rational points, note that for any $R:=(x,y)\in E(k)$, the quantity $$-p \frac{(x-a)(x - b)}{(x_0 - a)(x_0 - b)}$$ is always negative in at least one real place, since $x\in\Q$. Hence the conic $X_R$ has no $k$-rational point.

By \cite{ctps}*{Proposition 2.2}, $\phi^{*} \colon \Br(E)\to\Br(X)$ is surjective since all fibers of $\phi$ away from the fiber over the closed point of $E$ corresponding to $x=a$ have a geometrically integral component of multiplicity one defined over $k$. Since $\phi(\{M_v\})$ is orthogonal to $\Br(E)$ as it pairs in the same way as $P_0\in E(k)$, we have $\{M_v\}\in X(\A)^\Br$.

The remainder of the proof, that is, showing that $X(\A)^{\et,\Br} \ne \emptyset$, will follow from Theorem \ref{thm: ctpsthm} below, applied to the conic bundle $X' := X \times_E E \to E$ where $E \to E$ is given by translation by $P_0$. Since $E$ is a Serre curve \cite{Ser72}*{\S 5.2}, $E$ has large Galois image (see below). Since $X'$ is isomorphic to $X$ over $k$, $X'(\A)^{\et, \Br} \ne \emptyset$ implies $X(\A)^{\et,\Br} \ne \emptyset$, as desired. 
\end{proof}

In \cite{ctps}, an elliptic curve over $k$ is said to have \emph{large Galois image} if the image of $\rho\colon \Gal(\kbar/k)\to\GL_2(\Zhat)$ coming from the Galois representation of torsion points on $E$ contains $\SL_2^+(\Zhat)$, where $\SL_2^+(\Zhat)$ is the kernel of the composition $\SL_2(\Zhat)\to \GL_2(\Z/2\Z)\to \Z/2\Z$. The following theorem is an amalgamation of various results proved in \cite{ctps}.

\begin{thm} \label{thm: ctpsthm} Let $k$ be a totally real field, and let $K$ be a field such that $k \subset K \subset k^{\textup{cyc}}$. Suppose that $E$ is an elliptic curve over $k$ such that $E(\R)$ is connected and $E$ has large Galois image. Let $\phi \colon X \to E$ be a conic bundle where each fiber contains a geometrically integral component of multiplicity one, except a single nonsplit fiber over $P \in E(K)$ which is nondivisible. Finally, suppose that the fiber over $\calO \in E(k)$ is locally soluble at all finite places of $k$. If $\{M_v\}\in X(\A)$ where $\phi(M_v)=\calO$ for all finite places, then $\{M_v\}\in X(\A)^{\et,\Br}$.
\end{thm}

\begin{proof}
The proof follows from \cite{ctps}*{Thm.\ 5.6} mutatis mutandis, however, we repeat a sketch of the argument here for the convenience of the reader. By \cite{ctps}*{Prop.\ 2.2}, the map $\Br(E)\to \Br(X)$ is surjective. By \cite{ctps}*{Prop.\ 2.3}, any torsor $X'\to X$ comes from some torsor $\phi\colon E' \to E$. Twisting by a finite $k$-group scheme $G$, we can assume $E'$ has a point $\calO'$ mapping to $\calO$. Let $C\subset E'$ be the connected component of $\calO'$. Choosing $\calO'$ as the origin of $C$ makes $C\to E$ an isogeny of elliptic curves. Then $Y\subset X'\times C$ is a conic bundle $g\colon Y\to C$. By \cite{ctps}*{Thm.\ 4.1}, $\phi^{-1}(P)$ is integral. Hence $Q:=\phi^{-1}(P)$ is a closed point of $C$, so applying \cite{ctps}*{Prop.\ 2.2} again gives $g^*\colon\Br(C)\to \Br(Y)$ is surjective. For each finite place $v$, let $M_v'$ be a point on $Y$ on the fiber above $\calO'$ that maps to $M_v\in X(k_v)$. For the infinite places $v$, let $M_v'$ be any real point on $Y$ mapping into $M_v$. Since $\Br(Y)=g^*\Br(C)$, we have $\{M_v'\}\in Y(\A)^\Br$ so that $\{M_v'\}\in X'(\A)^\Br$. Hence $\{M_v\}\in X(\A)^{\et,\Br}$.
\end{proof}

\section{Distribution of rational points}\label{sec:distribution}

The results of the previous section, in particular Theorem \ref{Cor: BMO real quad}, demonstrated the insufficiency of the Brauer--Manin obstruction for conic bundles over positive rank elliptic curves. In this section we consider those varieties $X$ with a conic bundle structure $\phi\colon X\to E$ for which $X(k) \ne \emptyset$, and study properties of the image of $X(k)$ inside of $E(k)$.

For the remainder of this section, we work over $k =\Q$, and we assume that $\phi$ does not have a section over $\Q$ and $X(\Q)\neq\emptyset$. Since $X\to E$ is a conic bundle, the condition that a fiber contains a rational point is equivalent to the fiber being everywhere locally soluble. If instead we consider those fibers that are soluble for a finite set of primes, we always have the following result:

\begin{prop}\label{prop:finiteset}
For any finite set of primes $S$, let
\begin{equation} \label{eqn: Z_S} \calZ_S:=\{ P \in  E(\Q) \mid \textup{the fiber above $P$ contains $\Q_p$-points for $p\in S$} \} \subseteq E(\Q) 
\end{equation}
Then for any $P\in \calZ_S$, there is a finite index subgroup $\calH\subset E(\Q)$ such that $P+\calH\subset \calZ_S$.
\end{prop}

\begin{proof}
Let $P\in\calZ_S$. For each $p\in S$, choose a $\Q_p$-point $x \in  \phi^{-1}(P)$. Since $X$ is smooth, there is a $p$-adic analytic open subset $U_p\subset X(\Q_p)$ containing $x$. Since the image $E(\Q)\injects E(\Q_p)$ is locally dense, any open subset around the origin in $E(\Q_p)$ will contain a subgroup of finite index in $E(\Q)$. Hence, we have that $(\phi(U_p)-P)\cap E(\Q)$ contains a subgroup $\calH_p\subset E(\Q)$ of finite index. Setting $\calH=\cap_{p\in S} \calH_p$ gives us the desired finite index subgroup.
\end{proof}

Motivated by Proposition \ref{prop:finiteset}, we are interested in the following question concerning the set $\phi(X(\Q)) \subseteq E(\Q)$:

\begin{question}  \label{questions} Does $\phi(X(\Q))$ contain a translate of a finite index subgroup in $E(\Q)$, i.e., does the closure $\overline{\phi(X(\Q))}\subseteq E(\A)$ contain an open subset of $\overline{E(\Q)} \isom \hat{\Z}^r$? (see e.g., \cite{Sko99}*{p.\ 126})
\end{question}

Suppose $\phi\colon X\to E$ is a Ch\^atelet conic bundle given by
$$u^2-av^2=f(x)w^2,$$
where $a\notin\Q^{\times2}$ and $f(x)$ is separable. For any prime $p$, set
$$\calR_p:=\{P\in E(\Q)\mid \phi^{-1}(P)\text{ does not contain a }\Q_p\text{-point}\}.$$
For each prime $p$ of good reduction, denote $E_p$ the reduction of $E$ modulo $p$. Let $E_1(\Q_p):=\ker\left(E(\Q_p)\to E_p(\F_p)\right)$. Then any $P \ne \calO \in E_1(\Q_p)$ has $3 v_p(x(P)) = 2 v_p(y(P)) = -6i$ for some integer $i \ge 1$~\cite{silverman}*{VII.2 Prop.\ 2.2}. The following lemma illustrates the continuity of the addition map on $E(\Q_p)$.

\begin{lem}\label{lem:elladdition}
Let $E/\Q$ be an elliptic curve given by a short Weierstrass equation and $p$ be a prime of good reduction. Let $\calO\neq P\in E_1(\Q_p)$ and let $i \ge 1$ be the integer such that $v_p(x(P)) = -2i$. If $Q\in (w,z)\in E(\Q_p)$ where $w,z\in \Z_p^\times$, then $x(P+Q)=w+p^iu$ for some $u\in\Z_p^\times$. 
\end{lem}

\begin{proof}
Let $P=(\frac{x}{p^{2i}},\frac{y}{p^{3i}})$ where $x,y\in \Z_p^\times$. By the addition formula,
\begin{align*}
x(P+Q)&=\left(\frac{y-p^{3i}z}{p^ix-p^{3i}w}\right)^2-\frac{x}{p^{2i}}-w\\
&=\frac{1}{p^{2i}}\frac{y^2-x^3+2p^{2i}x^2w+p^{3i}u}{(x-p^{2i}w)^2}-w\\
&=\frac{1}{p^{2i}}\frac{2p^{2i}x^2w+p^{3i}u'}{(x-p^{2i}w)^2}-w\\
&=(2x^2w +p^i u')(x^{-2} + p^{2i}v) - w \\
&= w + p^i u''
\end{align*}
where $u,u',u'',v \in \Z_p^\times$.
\end{proof}

Let $\Omega$ denote the set of all primes $p$ such that $E$ has good reduction, $p\nmid a$, and $\calR_p\neq\emptyset$. In particular, this means $a\notin \Q_p^{\times2}$ for any $p\in \Omega$. For any prime $p$ of good reduction for $E$, let $\red_p\colon E(\Q)\to E_p(\F_p)$ be the reduction map.

\begin{prop}\label{prop:finindex}
The image $\phi(X(\Q))$ contains a translate of a finite index subgroup if and only if there exists a finite index subgroup $\calH\subset E(\Q)$ and $P\in\phi(X(\Q))$ such that $\red_p(P+\calH)\cap \red_p(\calR_p)=\emptyset$ for all but finitely many $p\in\Omega$.
\end{prop}

\begin{proof}
Suppose $\phi(X(\Q))$ contains $P+\calH$. We have $nE(\Q)\subset \calH$ for some integer $n$. Let $\overline{P}=\red_p(P)$. Suppose $p\in\Omega$ be a prime, not dividing $a,n$, such that $\overline{P+H}=\overline{R}$ for some $H\in\calH$ and $R\in\calR_p$. We can assume that $f(x)\bmod p$ is separable. Write $P=R+H+Q$ for some $Q\in E_1(\Q_p)$. Note that $Q\neq0$. By the theory of formal groups of elliptic curves, there is an isomorphism of groups
\begin{equation}\label{eqn:formalgroup}
    \xi\colon E_1(\Q_{p}) \xrightarrow{\sim} p\Z_p = \widehat{E}(p\Z_p), \quad
    (x,y) \mapsto -\frac{x}{y}
\end{equation}
such that $v_p(\xi(Q))=-v_p(x(Q))/2$~\cite{silverman}*{VII Prop.\ 2.2, Prop.\ 6.3}. Hence, for any $Q \in E_1(\Q_p)$ and any integer $m$, we have $v_p(\xi(mQ)) = v_p(m) + v_p(\xi(Q))$, and hence $v_p(x(mQ))/2 = -v_p(m) + v_p(x(Q))/2$. Let $x_0=x(R)$ and set $i=v_p(f(x_0))$ (note that $f(x_0)\neq0$ since $R\in\calR_p$). Note that $i$ is odd since $R\in\calR_p$. Let $j$ be an integer such that $j>i$, and write $p^j=1+tn$ for some $t\in\Z$. Then
$$P=R+H+Q\equiv R+p^jQ\bmod \calH$$
so there exists $H'\in\calH$ such that $P+H'=R+p^jQ$. Then $j'=-v_p(x(p^jQ))/2=j-v_p(x(Q))/2>j$ and choosing a different $j$ if necessary, we assume $j'$ is odd. Lemma \ref{lem:elladdition} then implies that $x_1=x(P+H')=x(R+p^jQ)=x(R)+p^ju$ for some $u\in\Z_p$. Hence $v_p(f(x_1))=v_p(f(x_0)+p^jf'(x_0)+p^{2j}u')$ for some $u'\in\Z_p$. It follows that $v_p(f(x_1))=i$, which is odd. Hence the fiber above $P+H'$ does not have a $\Q_p$-point which is a contradiction.

Conversely, suppose $\red_p(P+\calH)\cap\red_p(\calR_p)=\emptyset$ for all but finitely many $p\in\Omega$. Let $\Omega^0\subset\Omega$ be the set of primes for which this is true. Then it follows from definition of $\calR_p$ that any fiber above a point in $P+\calH$ contains $\Q_p$-points for $p\in\Omega^0$. Now setting $S=\Omega\setminus\Omega^0$, by Proposition \ref{prop:finiteset}, there exists a finite index subgroup $\calH'$ such that any fiber above a point in $P+\calH'$ is $\Q_p$-soluble for $p\in S$. Putting $\calH''=\calH\cap\calH'$, we see that $P+\calH''\subset\phi(X(\Q))$.
\end{proof}
 
Let us illustrate a case where Proposition \ref{prop:finindex} can be applied directly. Suppose that for any positive integer $n$ and finite set $S\subset \Omega$, there exists $p\in \Omega\setminus S$ such that
\begin{equation}\label{eqn:gcd1}
    \gcd(n,|E_p(\F_p)|)=1.
\end{equation}
Suppose that $\Omega$ is infinite. Then for any finite index subgroup $\calH\subset E(\Q)$, there exists infinitely many $p\in\Omega$ such that $\red_p(\calH)=\red_p(E(\Q))$. Hence, Proposition \ref{prop:finindex} would imply that $\phi(X(\Q))$ does not contain a translate of a finite index subgroup.

Suppose that $X\to E$ is ramified over a nonzero rational point $P\in E(\Q)$. For any $p$ of good reduction not dividing the discriminant of $f$, choose a point $Q\in E_1(\Q)$ such that $-v_p(x(Q))/2$ is odd.  Then by Lemma \ref{lem:elladdition}, $v_p(f(x(P+Q)))=-v_p(x(Q))/2$ is odd, so if $a$ is not a square modulo $p$, then $P+Q\in \calR_p$. Hence, $\Omega$ consists of all but finitely many primes $p$ such that $a$ is not a square modulo $p$. In \S\ref{sec:torsionpoints} we prove Corollary \ref{ellred}, which states that \eqref{eqn:gcd1} holds for such $\Omega$ under the following conditions:

\begin{enumerate}
    \item For any isogenous curve $E'$, $E'(\Q)$ is torsion free;
    \item The Galois representation $\rho_\ell\colon\Gal(\Qbar/\Q)\to\GL_2(\Z/\ell\Z)$ on the torsion points on $E$ is surjective for all $\ell>2$ except for possibly one prime;
    \item If $\rho_\ell$ is not surjective for some $\ell>2$, then $\Q(E[2])\cap \Q(E[\ell])=\Q$;
    \item $\sqrt{a}\notin \Q(E[2]),\Q(\sqrt{-3})$.
\end{enumerate}
Any elliptic curve with no exceptional prime (i.e.\ $\rho_\ell$ is surjective for every prime $\ell$) will satisfy the first 3 conditions (e.g., Serre curves). Hence, 100\% of elliptic curves (when ordered by height) satisfy these conditions ~\cites{duke97,jones10}. Note that condition (1) is necessary for \eqref{eqn:gcd1} to hold; see Remark \ref{rmk:conditions} for a discussion of the other conditions. The following theorem then follows immediately.

\begin{thm}\label{thm:arithprog}
Let $E/\Q$ be an elliptic curve of positive rank defined, and $\phi\colon X\to E$ be the Ch\^atelet conic bundle over $E$ given by
$$u^2-av^2=f(x)w^2.$$
Suppose conditions (1)-(4) above are satisfied. Suppose $(a,f(x))$ is ramified over a nonzero rational point in $E(\Q)$. Then $\phi(X(\Q))$ does not contain a translate of a finite index subgroup of $E(\Q)$.
\end{thm}

Now suppose there exists a positive integer $n$ and a finite subset $S\subset \Omega$ such that for all $p\in \Omega\setminus S$, we have
\begin{equation}\label{eqn:gcdd}
    \gcd(n,|E_p(\F_p)|)>1.
\end{equation}
This occurs, for example, when $E(\Q)$ has nontrivial torsion. To analyze whether $\phi(X(\Q))$ contains a translate of $\calH$, one would then need to analyze the sets $\red_p(\calH)$ and $\red_p(\calR_p)$. We show that it is possible for $\phi(X(\Q))$ to contain a translate of a finite index subgroup when $E(\Q)$ contains the full $2$-torsion subgroup.

Suppose that $E$ has positive rank and that $E(\Q)$ contains the full $2$-torsion subgroup. Let $P=(x_0,y_0)\in E(\Q)$ be a point of infinite order, not divisible by 2. Choose $Q\in E(\Qbar)$ such that $2Q=P$. Then $\Q(Q)$ is a Galois extension of degree $2$ or $4$, and is independent of the chosen point $Q$. Let $\Q(\sqrt{a})\subset \Q(Q)$ be a quadratic subfield. Choose points $R_i=(\alpha_i,\beta_i)\in E(\Q)\setminus E(\Q)[2]$ for $1\leq i\leq 2n$ where $n\in\Z$, such that $R_i+R_j\in 2E(\Q)$ for all $i,j$. Define the conic bundle $X\to E$ by
$$u^2-av^2=\prod_i (x-\alpha_i) w^2.$$
For $p\in\Omega$, since $a\notin\Q_p^{\times2}$, we have that $\overline{P}:=\red_p(P)$ is not divisible $2$. Let $\calH$ be the subgroup generated by $2P$. Hence the sets
$$\overline{\calH},\quad \overline{P+\calH}$$
are disjoint. On the other hand, $\overline{\calR_p}=\overline{\{R_1,\ldots,R_n\}}$. Hence $\overline{\calR_p}$ can intersect at most one of these sets, independent of $p$, since $R_i+R_j\in 2E(\Q)$ for any $i,j$. Proposition \ref{prop:finindex} then implies that $\phi(X(\Q))$ contains a translate of a finite index subgroup.

\section{Torsion points on elliptic curves}\label{sec:torsionpoints}

Let $E$ be an elliptic curve over $\Q$. Let $a\in\Q^\times\setminus\Q^{\times2}$ and $\Omega$ be the set of primes $p$ which are inert in $\Q(\sqrt{a})$. In this section, we give sufficient conditions for the following to hold:
For any positive integer $n$ and finite set $S\subset \Omega$, there exists $p\in \Omega\setminus S$ such that
\begin{equation}\label{eqn:gcd1sec4}
    \gcd(n,|E_p(\F_p)|)=1.
\end{equation}
The motivation for studying this is Theorem \ref{thm:arithprog}.

Note that an easy consequence of \eqref{eqn:gcd1sec4} is that $\gcd(n,|E'(\Q)_{\tors}|)=1$ for any isogenous elliptic curve $E'$ and all positive integers $n$. One can then ask whether this is sufficient, which can be considered as a local-to-global question on torsion points on elliptic curves.

In \cite{katz}*{Thm. 2}, Katz considered a similar problem, and proved that given any integer $n > 1$, $n\!\mid\! \#E_p(\F_p)$ for a set of primes $p$ of density one if and only if there exists another curve $E'$ $\Q$-isogenous to $E$ such that $n\!\mid\! \#E'(\Q)_{\tors}$. Note that \eqref{eqn:gcd1sec4} does not follow directly from Katz's result for two reasons. First, we are considering only primes in $\Omega$, which comprise a set of primes of density $1/2$. Second, $n$ is not necessarily a prime. For example if $n=\ell_1\ell_2$, we must consider the possibility that $E_p(\F_p)$ contains either $\ell_1$-torsion or $\ell_2$-torsion for almost all $p$.

We shall need the following notation. For any positive integer $n$, let $k_n$ denote the $n$th division field of $E$, i.e., $k_n:=\Q(E[n])$. Let $\rho_n\colon\Gal(k_n/\Q)\to\GL_2(\Z/n\Z)$ be the Galois representation on the $n$-torsion points. Let $NF_n\subset \Gal(k_n/\Q)$ be the subset of elements $\sigma$ such that $E[n]^\sigma=0$. If $L,M$ are Galois extensions of a field $K$ and $\sigma\in\Gal(L/K),\tau\in\Gal(M/K)$ then we say $\sigma$ and $\tau$ are \defi{compatible} if $\sigma|_{L\cap M}=\tau|_{L\cap M}$.

\begin{lem}\label{GaloisImage} Let $\ell>2$ be a prime and $G$ be the image of $\rho_\ell\colon \Gal(k_\ell/\Q)\to\GL_2(\F_\ell)$. Suppose there is a point $P\in E[\ell]$ such that $\det(\rho_\ell(\Gal(k_\ell/\Q(P))))=\F_\ell^\times$. Then $G$ is conjugate to one of the following groups:
\begin{enumerate} \label{GaloisImages}
\item  ${\tt G}: \GL_2(\F_\ell)$
\item  ${\tt B}: \left\{\begin{bmatrix}x&*\\0&y\end{bmatrix}\middle|\, x\in (\F_\ell^\times)^\frac{\ell-1}{d}, y\in \F_\ell^\times\right\}, d\mid \ell-1$ 
\item ${\tt Cs}: \left\{\begin{bmatrix}x&0\\0&y\end{bmatrix}\middle|\, x\in (\F_\ell^\times)^\frac{\ell-1}{d}, y\in \F_\ell^\times\right\}, d\mid \ell-1$ 
\item ${\tt B}: 
\left\{\begin{bmatrix}x&*\\0&y\end{bmatrix}\middle|\, x\in \F_\ell^\times, y\in (\F_\ell^\times)^\frac{\ell^2-\ell}{d}\right\}, \ell\mid d\mid\ell^2-\ell$
\item  ${\tt Ns}: \left\{\begin{bmatrix}x&0\\0&y\end{bmatrix},\begin{bmatrix}0&x\\y&0\end{bmatrix}\middle|\, x\in \F_\ell^\times,y\in \F_\ell^\times\right\}$
\item ${\tt S_4}: \ell=5$ and $G$ has projective image isomorphic to $S_4$ is exceptional. It is generated by
$\begin{bmatrix} 0&3\\3&4\end{bmatrix},\begin{bmatrix} 2&0\\0&2\end{bmatrix},\begin{bmatrix} 3&0\\4&4\end{bmatrix}$
\end{enumerate}
where $d=\deg(\Q(P)/\Q)$~\cite{Sut16}*{Table 3}.
\end{lem}

\begin{proof}
There are $\ell^2-\ell$ points in $E[\ell]\setminus \{O,P,2P,\ldots,(\ell-1)P\}$ which divide into Galois orbits over $\Q(P)$. Since $k_\ell/\Q(P)$ must have degree divisible by $\ell-1$, there is either one Galois orbit of size $\ell^2-\ell$ or there are $\ell$ orbits of size $\ell-1$.

\begin{case}[there exists an orbit of size $\ell^2-\ell$] Either the Galois closure of $\Q(P)$ is itself or it is $k_\ell$. Suppose $\Q(P)$ is Galois over $\Q$ of degree $d>1$. Let $Q$ be any point in the orbit of size $\ell^2-\ell$. With respect to the basis $\{P,Q\}$, $G$ is a subgroup of the matrix group
$$\left\{\begin{bmatrix}x&*\\0&y\end{bmatrix}\mid x\in (\F_\ell^\times)^\frac{\ell-1}{d}, y\in \F_\ell^\times\right\}.$$
Since $\deg(k_\ell/\Q)=\deg(k_\ell/\Q(P)) \cdot \deg(\Q(P)/\Q)=(\ell^2-\ell)d$, it follows that $G$ is in fact equal to the matrix group above.

Now suppose the Galois closure of $\Q(P)$ is $k_\ell$. Then we must have that $P$ is conjugate to every point in $E[\ell]\setminus\{O\}$. Hence \[ \deg(k_\ell/\Q)=\deg(k_\ell/\Q(P)) \cdot \deg(\Q(P)/\Q)=(\ell^2-\ell)(\ell^2-1).\] By order considerations, this forces $G$ to be $\GL_2(\F_\ell)$.
\end{case}

\begin{case}[there exist $\ell$ orbits of size $\ell-1$] 
Fix a basis $\{P,Q'\}$ of $E[\ell]$, such that a generator of $\Gal(k_\ell/\Q(P))$ is mapped to $\begin{bmatrix} 1 & x \\ 0 & y \end{bmatrix}$ for some $x \in \F_{\ell}$ and $y \in \F_\ell^\times$, $y \ne 1$. Then since $(\frac{x}{y-1},\, 1)$ is an eigenvector with eigenvalue $y$, we have that the orbit of $Q := \frac{x}{y-1}P + Q'$ is $\{Q,2Q,\ldots,(\ell-1)Q\}$. Hence, over $\Q(P)$, the points of $E[\ell] \setminus \{\calO, P, \dots, (\ell-1) P\}$ break up into the following $\ell$ orbits of size $\ell -1$:
\begin{align*}
O_0&=\{ Q, 2Q, \dots, (\ell - 1) Q \} \\
O_1&=\{ P + Q, P + 2Q, \dots, P + (\ell - 1) Q \} \\ 
\qquad &\vdots \qquad \\
O_{\ell-1}&=\{ (\ell-1)P + Q, \dots, (\ell-1)P + (\ell - 1) Q\}
\end{align*}

Suppose first that $\Q(P)$ is Galois over $\Q$, of degree $d>1$. With respect to the basis $\{P,Q\}$, $G$ is a subgroup of the matrix group
$$\left\{\begin{bmatrix}x&0\\0&y\end{bmatrix}\mid x\in (\F_\ell^\times)^\frac{\ell-1}{d}, y\in \F_\ell^\times\right\}.$$
By degree considerations again, we must have that $G$ is equal to the group above.

For the remainder of the proof, suppose the Galois closure of $\Q(P)$ is $k_\ell$. First suppose that $O_0$ is a Galois orbit over $\Q$. Let $T:=\{1\leq n\leq \ell-1\mid P\sim nP\}$, where $\sim$ denotes Galois conjugate points. There exists some $0<i\leq \ell-1$ such that $P$ is conjugate to points in $O_i$. Then $$\bigcup_{n\in T}O_{in}\cup\{nP\}$$
is a Galois orbit over $\Q$. From this we find that $\ell\mid d\mid\ell(\ell-1)$. Hence by degree count, $G$ must be conjugate to the matrix group
$$\left\{\begin{bmatrix}x&*\\0&y\end{bmatrix}\mid x\in \F_\ell^\times, y\in (\F_\ell^\times)^\frac{\ell^2-\ell}{d}\right\}.$$

Now suppose that $O_0$ is not a Galois orbit over $\Q$, i.e., $\Q(Q)$ is not Galois over $\Q$. Suppose $Q \sim aP + bQ$ for some $a\neq0$. If $b=0$, there exists some suitable $e$ such that $Q\sim eQ\sim eaP=P$, so $Q\sim P$. If $b>0$, then $Q \sim aP + bQ \sim a(P + Q)$, since the latter two are both in the orbit $O_a$. Since the Galois closure of $k(P)$ is $k_\ell$, $P$ must be conjugate to some point $cP+dQ$ where $d\neq0$. But then for some suitable $e$, we have $Q\sim eQ\sim ea(P+Q)= c(P+Q)\sim cP+dQ\sim P$.

Since all multiples of $Q$ are Galois conjugate, we have that the orbit of $Q$ over $\Q$ must at least contain
\begin{equation} \label{eqn: orbQ} \{ Q, \dots, (\ell-1)Q, P, \dots, (\ell-1)P \}. \end{equation} 

If this is not the whole orbit, then we claim that in fact the orbit must contain all points of $E[\ell] \setminus \calO$. To see this, note that if the orbit is larger than \eqref{eqn: orbQ}, then $Q \sim aP + bQ \sim a(P + Q)$, since the latter two are both in the orbit $O_a$. Thus there exists a multiple of $Q$ conjugate to $P+Q$, whence $Q \sim P+Q$. Finally, for any $c,d \ge 1$, we have $Q \sim cQ \sim c(P+Q) \sim cP + dQ$, as desired. Thus $|G| = (\ell^2 -1)(\ell - 1)$. The classification of subgroups of $\GL_2(\F_\ell)$ of order prime to $\ell$ tell us that the only possibilities for $G$ are the following:
\begin{enumerate}[(i)]
	\item $G$ is contained in a Cartan subgroup
	\item $G$ is contained in the normalizer of a Cartan group
	\item $G$ is exceptional.
\end{enumerate}
When $\ell > 3$, (i) and (ii) are impossible by order considerations; since a non-split Cartan subgroup has order $\ell^2 -1$, a split Cartan has order $(\ell -1)^2$, and any Cartan subgroup has index $2$ in its normalizer. Hence we are in the case of (iii), and the possibilities are either {\tt A4, S4}, or {\tt A5}. By order considerations, the only case possible is {\tt S4} and $\ell=5$. By \cite{Sut16}, the group listed in (6) occurs, e.g., for the elliptic curve $y^2=x^3+9x-18$. By \cite{Sut16}*{Lemma 3.21}, this is the only conjugacy class for $G$ that could occur.

Finally, suppose that \eqref{eqn: orbQ} is precisely the orbit of $Q$ over $\Q$. Then the remaining $\ell-1$ orbits over $\Q(P)$ must combine to give orbits of size $n(\ell -1)$ over $\Q$, where $n \mid (\ell -1)$. (The orbits must all have the same size since each multiple of $P+Q$ generates the same degree extension of $\Q$.) We shall see that $n = \ell +1$. So, consider the elements of $G$ given by the following matrices with respect to the basis $\{P,Q\}$:
\begin{equation}\label{eqn:matrices}
\begin{bmatrix} x & a \\ 0 & b \end{bmatrix}, \quad \begin{bmatrix}
1 & 0 \\ 0 & y
\end{bmatrix}.
\end{equation} 
The first exists as $P$ and every nonzero multiple of $P$ are in the same orbit over $\Q$. The latter exists since there are elements of $G$ fixing $P$. In the former, $a$ a priori depends on $x$. 

Then consider the product which sends $P + Q \mapsto (x + ay)P + (by)Q$. If $a \ne 0$ for some $x$, then there exists $y$ such that $x + ay \equiv 0 \pmod{\ell}$. Thus $P+Q \mapsto (by)Q$, but this is impossible as the orbit of $Q$ was assumed to be \eqref{eqn: orbQ}. Thus $a = 0$ for all $x$, so that $P + Q \sim xP + byQ$, whence $n = \ell -1$ by the same arguments as above. Note that then $\#G= 2(\ell - 1)^2$, since 
$\#G = [ k_\ell : \Q(P)] \cdot [\Q(P) : \Q] = (\ell -1) \cdot 2(\ell -1)$, where the latter is the size of the orbit of $P$ over $\Q$. Moreover, since since the orbit of $Q+P$ has size $(\ell - 1)^2$, there must be an element of $G$ of order $2$ which fixes $P + Q$.  We claim that such an element is of the form $\begin{bmatrix} 0 & 1 \\ 1 & 0 \end{bmatrix}$. This would prove that $G$ is the normalizer of a split Cartan subgroup, giving (5). Any element of order $2$ of $\GL_2(\F_\ell)$ fixing $P+Q$ must have the form $\begin{bmatrix} w & 1-w \\ 1+w  & -w \end{bmatrix}$ for some $w \in \F_\ell$. If $w \ne 0$, then for any $x, y \ne 0$, $G$ will contain elements of the form 
\[ 
\begin{bmatrix}
xw & x(1-w) \\ y(1+w) & -yw  
\end{bmatrix}, \quad  
\begin{bmatrix} 
xw & y(1-w) \\ x(1+w) & -yw 
\end{bmatrix}.
\]
There are $(\ell-1)^2$ such elements of the first kind, and for $x \ne y$, there are $(\ell - 1)(\ell-2)$ elements of the second kind. Since these elements are all distinct from \eqref{eqn:matrices}, this gives a contradiction as the size of $G$ is too large. \qedhere
\end{case}
\end{proof}

We now collect a few group theoretic lemmas which we will need in order to prove our main results.

\begin{lem}[\cite{Jordan}] \label{lem:sl2simple}
Let $\F_q$ be a finite field with $q>3$, the only proper normal subgroups of $\SL_2(\F_q)$ are $\{I\},\{\pm I\}$.
\end{lem} 

\begin{lem}\label{lem:abext} Suppose $\rho_n$ is surjective for some odd square-free integer $n>0$. If $K\subseteq k_n$ such that $K/\Q$ is an abelian extension, then $K\subseteq\Q(\zeta_n)$.
\end{lem}

\begin{proof}
Let $H\subset \GL_2(\Z/n\Z)$ be any normal subgroup with abelian quotient corresponding to $\Gal(K/\Q)$. Recall that $\det(\rho_n) = \chi_n$, the cyclotomic character, hence for any $\sigma \in \Gal(k_n/\Q)$ we have $\sigma(\zeta_n) = \zeta_n^{\det(\rho_n(\sigma))}$. Thus, it suffices to show that $H$ contains $\SL_2(\Z/n\Z)$.  As $n$ is square-free,
$$\SL_2(\Z/n\Z)\isom\prod_{p\mid n}\SL_2(\Z/p\Z),$$
so we may assume that $n=p$ is prime. For $p>3$, the only abelian quotient of $\SL_2(\Z/p\Z)$ is the trivial group by Lemma \ref{lem:sl2simple} and the fact that $\PSL_2(\Z/p\Z)$ is non-abelian, so we are done. For $p=3$, the commutator subgroup of $\GL_2(\Z/3\Z)$ is $\SL_2(\Z/3\Z)$, so it follows that $H$ contains $\SL_2(\Z/3\Z)$.
\end{proof}

\begin{prop}\label{prop:surj}
Let $K/\Q$ be an abelian extension and suppose that $n$ is an odd square-free integer such that $\rho_n$ is surjective. Then for any $\tau\in \Gal(K/\Q)$, there exists $\sigma\in NF_n$ such that $\sigma$ is compatible with $\tau$.
\end{prop}

\begin{proof}
Let $L:=K\cap k_n$. Let $a\in (\Z/n\Z)^\times$ be such that $\zeta_n\to\zeta_n^a$ is compatible with $\tau$.
Let $\sigma\in\Gal(k_n/\Q)$ be so that $\rho_n(\sigma)$ is the following matrix,
$$\begin{bmatrix}a&-a\\1&0\end{bmatrix}.$$
Then it is straightforward to check that $\sigma\in NF_n$ and $\det(\rho(\sigma))=a$. By applying Lemma \ref{lem:abext} to $L$, we have $\sigma$ and $\tau$ are compatible.
\end{proof}

\begin{lem}\label{lem:gl2quot}
Let $\ell$ be an odd prime. Then any proper normal subgroup $H\subseteq \SL_2(\Z/\ell\Z)$ has index divisible by $\ell$. In particular, if $H'\subseteq \GL_2(\Z/\ell\Z)$ is a proper normal subgroup with $\det(H')=(\Z/\ell\Z)^\times$ then its index is divisible by $\ell$.
\end{lem}

\begin{proof}
For $\ell>3$, the result follows from Lemma \ref{lem:sl2simple} and the fact that $\ell\mid \#\SL_2(\Z/\ell\Z)$.
That is, $H'$ cannot contain $\SL_2(\Z/\ell \Z)$ or else the determinant would fail to be surjective. So, consider $H' \cap \SL_2(\Z/\ell \Z)$. It is a normal subgroup properly contained in $\SL_2(\Z/\ell\Z)$ hence by Lemma \ref{lem:sl2simple}, it is either trivial or $\{\pm{I}\}$. Thus $\# H' \mid 2(\ell -1)$, and the index is therefore divisible by $\ell$. For $\ell=3$, the result follows from the classification of subgroups of $\SL_2(\Z/3\Z)$.
\end{proof}

\begin{cor}\label{cor:surj}
Let $\ell_1,\ldots, \ell_s$ be distinct odd primes and let $n = \ell_1 \cdots \ell_s$. Suppose that $\rho_{E,\ell_i}$ is surjective for each $i$. Then $\rho_{E,n}$ is also surjective. Equivalently, $k_m\cap k_n=\Q$ for any odd relatively prime integers $m,n$ such that $\rho_m,\rho_n$ are surjective.
\end{cor}

\begin{proof}
Without loss of generality, assume that $\ell_i<\ell_j$ for $i<j$. Let $1\leq i<s$ and $m=\ell_1\cdots \ell_{i}$, and suppose that $\rho_{E,m}$ is surjective. We show that the fields $k_{m}$ and $k_{\ell_{i+1}}$ intersect trivially over $\Q$. By assumption, $\Gal(k_m/\Q)\isom \GL_2(\Z/m\Z)$ and $\Gal(k_{\ell_{i+1}}/\Q)\isom \GL_2(\Z/\ell_{i+1}\Z)$. Let $K=k_m\cap k_{\ell_{i+1}}$ and let $H\subset\GL_2(\Z/\ell_{i+1}\Z)$ be the normal subgroup such that $\GL_2(\Z/\ell_{i+1}\Z)/H\isom \Gal(K/\Q)$.

If $K\cap \Q(\zeta_{\ell_{i+1}})\neq\Q$ then Lemma \ref{lem:abext} would imply $K\cap \Q(\zeta_{\ell_{i+1}})\cap \Q(\zeta_m)\neq\Q$ which is a contradiction. Hence $K\cap \Q(\zeta_{\ell_{i+1}})=\Q$, so then $\det(H)=(\Z/{\ell_{i+1}}\Z)^\times$. If $H$ is proper, then by Lemma \ref{lem:gl2quot}, $\ell_{i+1}\mid [\GL_2(\Z/\ell_{i+1}\Z):H]=|\Gal(K/\Q)|$. However since $\GL_2(\Z/m\Z)$ has order
$$\prod_{1\leq j \leq i} \ell_j(\ell_j-1)^2(\ell_j+1),$$
its quotient cannot have order divisible by $\ell_{i+1}$. Hence we must have $H=\GL_2(\Z/\ell_{i+1}\Z)$ and so $K=\Q$. Thus
$$\Gal(k_{\ell_1\cdots \ell_{i+1}}/\Q)=\Gal(k_m/\Q)\times\Gal(k_{\ell_{i+1}}/\Q)\isom \GL_2(\Z/(\ell_1\cdots\ell_{i+1})\Z).$$
It follows then that $\rho_{m\ell_{i+1}}$ is also surjective.
\end{proof}

\begin{cor}\label{cor:k_2intk_n}
Let $K/\Q$ be an abelian extension and $n$ a square-free odd integer. Suppose that $\rho_n$ and $\rho_2$ are surjective. If $k_2\subset k_nK$, then $3\mid n$ and $k_2\subset k_3$.
\end{cor}

\begin{proof}
By Lemma \ref{lem:abext}, $K\cap k_n\subset \Q(\zeta_n)$, so we have the exact sequence
\[\begin{tikzcd}
1  \ar{r} & \Gal(k_n/\Q(\zeta_n)) \ar{r} & \Gal(k_nK/\Q) \ar{r} & \Gal(K(\zeta_n)/\Q) \ar{r} & 1
\end{tikzcd},\]
which induces
\[\begin{tikzcd}[column sep = 1.65 em]
1  \ar{r} & \Gal(k_2/\Q(\zeta_n)\cap k_2) \ar{r} & \Gal(k_2/\Q) \ar{r} & \Gal(K(\zeta_n)\cap k_2/\Q) \ar{r} & 1
\end{tikzcd}.\]
Now
$$\Gal(k_2/\Q(\zeta_n)\cap k_2)\isom \prod_{\substack{\ell\mid n\\ \ell\text{ prime}}}\Gal(k_\ell\cap k_2/\Q(\zeta_\ell)\cap k_2)$$
is a normal subgroup of $\Gal(k_2/\Q)\isom S_3$. By Lemma \ref{lem:gl2quot}, this forces each factor in the product to be trivial except possibly when $\ell=3$. Since at least one factor must be nontrivial, otherwise $S_3\isom \Gal(K(\zeta_n)\cap k_2/\Q)$ would be abelian, we must have $3\mid n$. Since $\Gal(k_3\cap k_2/\Q(\zeta_\ell)\cap k_2)$ must simultaneously be a normal subgroup of $S_3$ and a quotient of $\Gal(k_3/\Q(\zeta_3))\isom \SL_2(\F_3)$, the only possibility is that it is isomorphic to $\Z/3\Z$. Finally we have $\Gal(k_2\cap k_3/\Q)$ is a quotient of $\Gal(k_2/\Q)\isom S_3$ and contains normal subgroup of order $3$. The only possibility is $\Gal(k_2\cap k_3/\Q)\isom S_3$, i.e., $k_2\subset k_3$.
\end{proof}

\begin{cor}\label{cor:nonfix2n}
Assume the hypotheses of Corollary \ref{cor:k_2intk_n}. Then for any $\tau\in \Gal(K(\zeta_3)/\Q)$ such that $\zeta_3^\tau=\zeta_3$, there exists $\sigma\in NF_{2n}$ such that $\sigma$ is compatible with $\tau$.
\end{cor}

\begin{proof}
By Corollary \ref{cor:k_2intk_n}, we have $3\mid n$ and $k_2\subset k_3$. The unique normal subgroup $H$ of order 8 inside $\GL_2(\F_3)\isom \Gal(k_3/\Q)$ such that $\GL_2(\F_3)/H\isom S_3$ is generated by
$$\left\langle\begin{bmatrix}0&1\\-1&0\end{bmatrix},\begin{bmatrix}-1&-1\\-1&1\end{bmatrix}\right\rangle.$$

A computation shows that the matrix
$$\begin{bmatrix}-1&1\\0&-1\end{bmatrix}$$
maps to an element of order 3 in $S_3\isom \Gal(k_2/\Q)$. Choose $\sigma'\in \Gal(k_3/\Q)$ so that $\rho(\sigma')$ is the matrix above. Then a computation shows $\sigma'\in NF_3$ and $\sigma'|_{k_2}\in NF_2$. Let $a\in (\Z/n\Z)^\times$ be such that $\zeta_n\to\zeta_n^a$ is compatible with $\tau$. Since $\zeta_3^\tau=\zeta_3$, the image of $a$ under $(\Z/n\Z)^\times\to(\Z/3\Z)^\times$ is trivial. Let $\sigma=(\sigma_1,\sigma_2)\in\Gal(k_{2n}/\Q)\isom \Gal(k_3/\Q)\times\Gal(k_{n/3}/\Q)$ be so that $\rho_{2n}(\sigma_1,\sigma_2)$ is
$$\left(\begin{bmatrix}-1&1\\0&-1\end{bmatrix},\begin{bmatrix}a&-a\\1&0\end{bmatrix}\right).$$
It is straightforward to check that $\sigma\in NF_{2n}$ and $\det(\rho(\sigma))=a$. By Lemma \ref{lem:abext}, $\sigma$ and $\tau$ are compatible.
\end{proof}

We now give the main results of this section on some local-to-global properties of torsion points of elliptic curves. We first consider the  case of a single prime $\ell > 3$, and then treat $2$ and $3$ separately. We then prove the general case of a composite integer $n$. For the remainder of the section, $K/\Q$ will denote a nontrivial Galois extension, which we sometimes additionally require to be abelian. To prove Theorem \ref{thm:arithprog}, we only need to consider the case when $K$ is a quadratic extension, but the results will hold for a general $K$.

\begin{prop}\label{torprime} Let $K/\Q$ be a nontrivial Galois extension and $\ell>3$ a prime. The following are equivalent
\begin{enumerate}
    \item $E(\F_p)[\ell]\neq0$ for all but finitely many primes $p$ that do not split completely in $K$.
    \item There exists a curve $E'$ isogenous to $E$ such that $E'(\Q)[\ell]\neq0$.
\end{enumerate}
\end{prop}

\begin{proof}
The implication (2)$\implies$(1) is immediate. Assume (1) is true. Let $\sigma\in\Gal(K(\zeta_\ell)/\Q)$ be such that $\sigma|_{\Q(\zeta_\ell)}$ generates $\Gal(\Q(\zeta_\ell)/\Q)$ and $\sigma$ does not act as the identity on $K$. Let $p$ be a prime of good reduction for $E$ such that $\sigma$ is in the conjugacy class of $\Frob_p$ and $E(\F_p)[\ell]\neq 0$. Let $P\in E[\ell]$ so that the reduction mod $p$, $\red_p(P)\in E(\F_p)[\ell]$. Hence there is an embedding $\Q(P)\injects \Q_p$. Since the only subfield of $\Q(\zeta_\ell)$ with an embedding into $\Q_p$ is $\Q$, it follows that $\Q(P)\cap\Q(\zeta_\ell)=\Q$. Hence the image of $\rho_\ell$ is one of the groups listed in Lemma \ref{GaloisImage}. Suppose that we are not in the case of group \hyperref[GaloisImages]{(4)} with $d=\ell$. Then, since $\ell>3$, one can check that $G$ is generated by the subset
$$J:=\{g\in G\mid \det(g-1)\neq0\}.$$
In particular, there must be some element $g\in J$ such that $g=\rho(\tau)$ for some $\tau\in\Gal(k_\ell K/\Q)$ which acts nontrivially on $K$. Moreover $\tau|_{k_\ell}\in NF_\ell$ since $\rho(\tau)\in J$. Let $p\neq\ell$ be any prime of good reduction such that $\tau$ is in the conjugacy class of $\Frob_p$. Then $\Frob_p$ does not fix any nontrivial $\ell$-torsion point in $E_p(\overline{\F_p})[\ell]$, so $E_p(\F_p)[\ell]=\{O\}$ and $p$ does not split completely in $K$ . Since there are infinitely many such primes $p$, we get a contradiction. That is, if (1) holds, then $G$ must be a Borel subgroup of the form \hyperref[GaloisImages]{(4)}.

Now suppose we are in the case of \hyperref[GaloisImages]{(4)} in Lemma \ref{GaloisImage} with $d=\ell$. Let $P,Q\in E[\ell]$ be the basis giving the matrix representation of \hyperref[GaloisImages]{(4)}. The subgroup generated by $P$ is Galois invariant, so there exists an elliptic curve $E'$ and an isogeny $\varphi\colon E\to E'$ whose kernel is the subgroup generated by $P$. Then $\varphi(Q)$ is fixed by $G_\Q$, so defines a rational $\ell$-torsion point on the isogenous curve $E'$.
\end{proof}

Proposition \ref{torprime} does not hold if we replace $\ell$ by 2 or 3. However, we can determine precisely when it fails as the next two propositions show.

\begin{prop}\label{prop:twotorsion}
Let $K/\Q$ be a nontrivial Galois extension. Suppose that for all but finitely many primes $p$ such that $p$ does not split completely in $K$, we have $E(\F_p)[2]\neq0$. Then either $E[2](\Q)\neq 0$ or $\Gal(k_2/\Q)=S_3$ and $K
\subset k_2$ has degree 2 over $\Q$.
\end{prop}

\begin{proof}
 Let $y^2 = f(x)$ be the short Weierstrass equation for $E$ over $\Q$, and suppose that $E[2](\Q)= 0$. Let $L=K\cap k_2$. It suffices to show there exists an element in $\Gal(k_2/\Q)$ which fixes no 2-torsion point and does not fix $K$, since otherwise there are infinitely many primes $p$ whose lift of the Frobenius equals this element giving a contradiction. If the splitting field of $f$ is a cyclic degree 3 extension, then $\Gal(k_2/\Q)$ is also cyclic. If $L=\Q$, $\Gal(k_2/\Q)=\Gal(K/\Q)\times\Gal(k_2/\Q)$ so choosing nontrivial elements in each component produces a Galois action which fixes no 2-torsion point and does not fix $K$. If $L=E[2]$, then we get the generator of $\Gal(k_2/\Q)$ does the job as well.

Now assume $\Gal(k_2/\Q)=S_3$.  If $L=\Q$, then the same argument as above applies. If $L/\Q$ has degree 3, Then the generator of $\Gal(L/\Q)$ fixes no 2-torsion point and does not fix $K$. If $L=k_2$, the same argument applies by choosing any element of order 3 in $\Gal(L/\Q)$. Hence $L/\Q$ must be of degree 2.
\end{proof}

\begin{prop}\label{prop:threetorsion}
Let $K/\Q$ be a nontrivial Galois extension. Suppose that for all but finitely many primes $p$ such that $p$ does not split completely in $K$, we have $E(\F_p)[3]\neq0$. Then either $E[3](\Q)\neq0$ or $K=\Q(\sqrt{-3})$.
\end{prop}

\begin{proof} Let $L=K\cap k_3$. As in the proof of Proposition \ref{prop:twotorsion}, if $K\neq L$ or $L=\Q$, then we get a contradiction. Hence assume $\Q\neq K\subset k_3$. Define
$$J:=\{g\in G\mid \det(g-1)\neq0\}$$
and let $H$ be the normal subgroup generated by $J$. Then $K\subseteq k_3^H$. By the same argument as in the beginning of Proposition \ref{torprime}, the image of $\rho$ is one of the non-exceptional subgroups in Lemma \ref{GaloisImage}. In the case $G=\GL_2(\F_3)$, it is straightforward to check that $G$ is generated by $J$, so $G=H$ which is a contradiction.

Since $\ell=3$, there are only three cases left to consider:
\setcounter{case}{0}
\begin{case}[{\tt B}]
$G$ is of the form
$$\left\{\begin{bmatrix}x&*\\0&y\end{bmatrix}\middle| \, x\in \F_3^\times, \, y\in \F_3^\times\right\}.$$
The set $J$ contains
$$\begin{bmatrix}-1&0\\0&-1\end{bmatrix},\begin{bmatrix}-1&1\\0&-1\end{bmatrix},\begin{bmatrix}-1&-1\\0&-1\end{bmatrix}$$
the above elements generate the index 2 subgroup $H=\{g\in G\mid \det(g)=1\}$. Hence $K\subset k_3^H=\Q(\sqrt{-3})$, so we must have $K=\Q(\sqrt{-3})$.
\end{case}

\begin{case}[{\tt Cs}]
$G$ is of the form
$$\left\{\begin{bmatrix}x&0\\0&y\end{bmatrix}\middle| \, x\in \F_3^\times,\, y\in \F_3^\times\right\}.$$
The set $J$ contains
$$\begin{bmatrix}-1&0\\0&-1\end{bmatrix}.$$
which again generates the index 2 subgroup $H=\{g\in G\mid \det(g)=1\}$, so $K=\Q(\sqrt{-3})$.
\end{case}

\begin{case}[{\tt Ns}]
$G$ is the normalizer of a split Cartan. The set $J$ contains
$$\begin{bmatrix}0&1\\-1&0\end{bmatrix}.$$
which generates a cyclic normal subgroup of order 4. One checks that this subgroup is precisely $H=\{g\in G\mid \det(g)=1\}$, so $K=\Q(\sqrt{-3})$.
\qedhere
\end{case} 

\end{proof}
\setcounter{case}{0}

We now prove the general case for a composite integer $n$, which is used to prove Theorem \ref{thm:arithprog}. We require additional assumptions on $K$ if $2$ or $3$ divides $n$, as demonstrated by the previous two propositions.

\begin{thm}\label{thm:torcomposite}
Let $K/\Q$ be a nontrivial abelian extension not contained in $k_2$ or $\Q(\zeta_3)$. Let $n>1$ be an integer. Suppose that for all except possibly one odd prime $\ell\mid n$, we have $\rho_{\ell}$ is surjective. If $\rho_{\ell}$ is not surjective and $2\ell \mid n$, assume $k_2\cap k_{\ell}=\Q$. Then the following are equivalent:
\begin{enumerate}
    \item $E(\F_p)[n]\neq0$ for all but finitely many primes $p$ that do not split completely in $K$.
    \item There exists a curve $E'$ isogenous to $E$ such that $E'(\Q)[n]\neq0$.
\end{enumerate}
\end{thm}

\begin{proof}
As before, $(2)\implies (1)$ is clear. Now assume that $(1)$ holds. Note that it suffices to consider when $n$ is a product of distinct primes. Furthermore, if there exists $\tau\in\Gal(k_nK/\Q)$ such that $\tau|_{k_n}\in NF_n$ and $K^\tau\neq K$, then for any prime $p$ of good reduction such that $\Frob_p$ is in the conjugacy class of $\tau$, we have $E_p(\F_p)[n]=0$ and $p$ does not split completely in $K$. Our proof will show that either (2) is true, or there exists such an element $\tau\in\Gal(k_nK/\Q)$, thereby contradicting $(1)$.

\begin{case}[$\rho_\ell$ is surjective for all odd $\ell$]
If $n$ is odd, then by Proposition \ref{prop:surj}, there exists $\tau\in\rho(\Gal(k_nK/\Q))$ such that $\tau|_{k_n}\in NF_n$ and $K^\tau\neq K$.

Now assume $n=2m$. If $m=1$, then our result follows from Proposition \ref{prop:twotorsion}, so assume $m>1$. Since $E(\Q)[2]=0$, we have $\Gal(k_2/\Q)\isom \Z/3\Z$ or $S_3$. Hence,
$\Gal(k_2\cap k_{m}/\Q)\isom \{1\},\Z/2\Z,\Z/3\Z,$ or $S_3$. Let $L:=k_2\cap k_{m}K$.

{\it If $L/\Q$ is abelian:} By hypothesis on $K$, there exists $s\in \Gal(LK/\Q)$ such that $K^\sigma\neq K$, and $L^\sigma=L$ if $L/\Q$ is at most quadratic or $L^\sigma=\Q$ if $L/\Q$ is cubic. By Proposition \ref{prop:surj}, there exists $\sigma\in\rho(\Gal(k_{m}LK/\Q))$ such that $\sigma_{k_m}\in NF_m$ and $\sigma$ restricts to $s$. Then there exists a lift $\sigma$ to $\tau\in \Gal(k_nK/\Q)$ that satisfies $\tau|_{k_n}\in NF_n$.

{\it If $L/\Q$ is $S_3$:}
By Corollary \ref{cor:k_2intk_n}, we have $3\mid m$ and $k_2\subset k_3$. By Lemma \ref{lem:abext}, the only quadratic extension contained in $k_3$, and hence in $k_2$, is $\Q(\zeta_3)$. Since $K\not\subset k_2$ by hypothesis, this implies $K\neq\Q(\zeta_3)$. Hence, Corollary \ref{cor:nonfix2n} shows there exists $\tau\in\Gal(k_{n}K/\Q)$ such that $\tau|_{k_{n}}\in NF_{n}$ and $K^\tau\neq K$.
\end{case}

\begin{case}[$\rho_\ell$ is not surjective for some $\ell$]
Let $\ell\mid n$ be a prime such that $\rho_\ell$ is not surjective.
\begin{lem}
The image of $\rho_\ell$ is given by one of the groups in Lemma \ref{GaloisImage}
\end{lem}
\begin{proof} Assume first that $2\nmid n$. By Proposition \ref{prop:surj}, there exists $\sigma\in \Gal(k_{n/\ell}K(\zeta_\ell)/\Q)$ such that $\sigma \in NF_{n/\ell}$ and $K^\sigma\neq K$, $\Q(\zeta_\ell)^\sigma=\Q$. Then for any prime $p$ of good reduction such that $\Frob_p$ is in the conjugacy class of $\sigma$, we have $E_p(\F_p)[n/\ell]=0$ and $p$ does not split completely in $K$. Hence by (1), $E_p(\F_p)$ contains an $\ell$-torsion point, i.e., $\sigma$ fixes an $\ell$-torsion point on $E$. Since $\sigma$ also generates $\Gal(\Q(\zeta_\ell)/\Q)$, the hypothesis of Lemma \ref{GaloisImage} are satisfied.

Now assume $n=2\ell m$ for some integer $m$ where $\ell$ again is an odd prime such that $\rho_\ell$ is not surjective. Set $L:=k_2\cap k_mK(\zeta_\ell)$ and recall that
$\Gal(L/\Q)\isom \{1\},\Z/2\Z,\Z/3\Z,$ or $S_3$.

{\it If $L/\Q$ is abelian:} By hypothesis on $K$, there exists $s\in \Gal(LK(\zeta_\ell)/\Q)$ such that $K^s\neq K$ and $=\Q(\zeta_\ell)^s=\Q$ and $L^s=L$ if $L/\Q$ is quadratic or $L^s=\Q$ if $L/\Q$ is cubic. By Proposition \ref{prop:surj}, there exists $\sigma\in\Gal(k_{m}KL(\zeta_\ell)/\Q)$ such that $\sigma|_{k_{m}}\in NF_{m}$ and $\sigma$ restricts to $s$. By construction, any lift of $\sigma$ to $\tau\in \Gal(k_{2m}K(\zeta_\ell)/\Q)$ satisfies $\tau|_{k_2}\in NF_2$.

{\it If $L/\Q$ is $S_3$:} By Corollary \ref{cor:k_2intk_n}, we have $3\mid m$ and $k_2\subset k_3$. By Lemma \ref{lem:abext}, the only quadratic extension contained in $k_3$, and hence in $k_2$, is $\Q(\zeta_3)$. Recall $K\cap k_2=\Q(\zeta_\ell)\cap k_2=\Q$ by hypothesis. Hence any $s\in\Gal(K(\zeta_\ell)/\Q)$ such that $K^s=\Q(\zeta_\ell)^s=\Q$ can be lifted to fix $\zeta_3$.  Hence, Corollary \ref{cor:nonfix2n} shows there exists $\tau\in\Gal(k_{2m}K(\zeta_\ell)/\Q)$ such that $\tau|_{k_{2m}}\in NF_{2m}$ and $K^\tau\neq K$ and $\Q(\zeta_\ell)^\tau=\Q$.

In either case, the claim is proved by the same argument as in the case $2\nmid n$ using $\tau$ to produce primes $p$ with $\ell$-torsion in the reduction mod $p$.
\end{proof}
\end{case}

Now we resume the proof of Theorem \ref{thm:torcomposite}. We consider two cases based on whether the image of $\rho_\ell$ is exceptional or not.

\begin{subcase}[The image of $\rho_\ell$ is not exceptional] Let $k_\ell'$ be the maximal abelian subextension of $k_\ell$. We can identify under $\rho_\ell$,
$$\Gal(k_{\ell}'/\Q)\isom \Z/m\Z\times \Z/(\ell-1)\Z$$
where $1 \leq m\leq\ell-1$. For \hyperref[GaloisImages]{(1)}$-$\hyperref[GaloisImages]{(4)}, this is clear. In the case of \hyperref[GaloisImages]{(5)}, $m=2$ and the Galois group is isomorphic to the quotient by the normal subgroup
$$\left\{\begin{bmatrix} x&0\\0&x^{-1}\end{bmatrix}\middle| x\in \F_\ell^\times\right\}.$$
The only case in which $m=1$ is possible is for \hyperref[GaloisImages]{(4)}, and $\rho_\ell(\Gal(k_\ell/\Q))$ is conjugate to
$$\left\{\begin{bmatrix}x&*\\0&1\end{bmatrix}\middle|\, x\in \F_\ell^\times\right\}.$$
Then $E$ will have an isogeny of degree $\ell$ onto an elliptic curve $E'$ having an $\ell$ torsion point, giving (2).

Now assume $m>1$. Let $N\subset\Gal(k_{\ell}'/\Q)\isom \Z/m\Z\times \Z/(\ell-1)\Z$ be the subset consisting of elements of the form $(a,b)$ where $a,b\neq0$. Then any lift of an element $g\in N$ to $\Gal(k_\ell/\Q)$ lies in $NF_\ell$. Again this is clear for \hyperref[GaloisImages]{(1)}$-$\hyperref[GaloisImages]{(4)}, and in the case of \hyperref[GaloisImages]{(5)}, the lift consists of matrices of the form
$$\left\{\begin{bmatrix} 0&x\\y&0\end{bmatrix}\middle| x,y\in \F_\ell^\times, xy\neq1 \right\}$$
which does not have any nontrivial eigenvector. Note that if $\ell>3$, then $N$ generates the entire group $\Gal(k_\ell'/\Q)$. On the other hand, if $\ell=3$, then $K\not\subset k_\ell$. Hence there exists $\sigma\in \Gal(k_\ell'K/\Q)$ such that $\sigma|_{k_\ell'}\in N$ and $K^\sigma=\Q$.

If $2\nmid n$, then by Proposition \ref{prop:surj}, there exists $\tau'\in\Gal(k_{n/\ell}k_\ell'K/\Q)$ such that $\tau'|_{k_{n/\ell}}\in NF_{n/\ell}$ and $\tau$ restricts to $\sigma$. Any lift $\tau\in\Gal(k_nK/\Q)$ will then satisfy $\tau\in NF_n$ and $K^\tau\neq K$.

Now suppose $n=2\ell m$. Let $M:=k_\ell'k_mK$ and $L:=k_2\cap M$.

{\it If $L/\Q$ is abelian:} 
Choose $u\in\Gal(L/\Q)$ to be the identity if $L/\Q$ is quadratic or a generator if $L/\Q$ is cubic. By the hypothesis $k_2\cap k_\ell=\Q$, we have $\Gal(Lk_\ell'/\Q)\isom \Gal(L/\Q)\times \Gal(k_\ell'/\Q)$. Moreover if $K\subset Lk_\ell'$, then the assumption $K\cap k_2=\Q$ and the fact that $N$ generates $\Gal(k_\ell'/\Q)$ means there exists some $\sigma_1=(s,t)\in \Gal(L/\Q)\times\Gal(k_\ell'/\Q)$ such that $s=u$, $t\in N$, and $K^{(s,t)}\neq K$. By Proposition \ref{prop:surj}, there exists $\sigma_2\in\Gal(k_{m}k_\ell'K/\Q)$ such that $\sigma_2$ is compatible with $\sigma_1$. Let $\sigma_3\in\Gal(k_{n/\ell}k_\ell'K/\Q)$ be a lift of $\sigma_2$ so that $\sigma_3|_{k_2}$ has order 3, i.e., contained in $NF_2$ (this is possible since $k_2\cap k_{m}k_\ell'K=L$). Finally any lift $\tau\in\Gal(k_{n}K/\Q)$ of $\sigma_3$ satisfies $\tau\in NF_n$ and $K^\tau\neq K$.

{\it If $L/\Q$ is $S_3$:} 
By Corollary \ref{cor:k_2intk_n}, we have $3\mid m$ and $k_2\subset k_3$. The hypothesis $k_2\cap k_\ell=\Q$ implies any $s\in\Gal(k_\ell'K/\Q)$ with $s|_{k_\ell'}\in N$ and $K^s=\Q$ is compatible with $1\in \Gal(\Q(\zeta_3)/\Q)$. Hence, Corollary \ref{cor:nonfix2n} implies there exists $\sigma\in\Gal(k_{n/\ell}k_\ell'K/\Q)$ with $\sigma\in NF_{n/\ell}$ and $K^\tau=\Q$. Then any lift $\tau\in\Gal(k_{n}K/\Q)$ of $\sigma$ satisfies $\tau\in NF_\ell$.
\end{subcase}

\begin{subcase}[Exceptional image] Now assume that $\ell=5$ and the image is exceptional as in \hyperref[GaloisImages]{(6)} of Lemma \ref{GaloisImage}. Recall that $\rho_5(G_\Q)$ generated by
$$A=\begin{bmatrix} 0&3\\3&4\end{bmatrix},B=\begin{bmatrix} 2&0\\0&2\end{bmatrix},C=\begin{bmatrix} 3&0\\4&4\end{bmatrix}$$
Let $k_5'\subset k_5$ be the subextension corresponding to the image of $\rho_5$ in $\PGL_2(\F_5)$. Then $\Gal(k_5'/\Q)\isom S_4$. The possible quotients of $S_4$ are $S_4,S_3,\Z/2\Z,\{0\}$. We will prove the desired result under the assumption that $6\mid n$, which is the most difficult case. If $2$ and/or $3$ does not divide $n$, the argument is easier, and follows the same pattern as the case $6\mid n$, so we do not include it here. By assumption $\rho_3(G_\Q)\isom \GL_2(\F_3)$ and the possible quotients are $\GL_2(\F_3),\PGL_2(\F_3),S_3,\Z/2\Z,\{0\}$. Recall that $\PGL_2(\F_3)\isom S_4$, so the possiblities for $\Gal(k_5'\cap k_3)$ are (1) $S_4\isom \PGL_2(\F_3)$, (2) $S_3$, (3) $\Z/2\Z$, (4) $\{0\}$.

We claim that $k_3\cap k_5'=\Q$. Note that $\Q(\zeta_5+\zeta^{-1}_5)\subset k_5'$ is the unique quadratic field inside $k_5'$ and $\Q(\zeta_3)\subset k_3$ is the unique quadratic field inside $k_3$. But by the discussion above, if $k_3\cap k_5'\neq\Q$, then there exists a quadratic field $J\subset k_3\cap k_5'$, which is a contradiction.

Next we claim $k_2\cap k_3k_5'=\Q,\,\Q(\zeta_3),\,\Q(\zeta_3(\zeta_5+\zeta_5^{-1})),$ or $k_2\subset k_3$ with $\Gal(k_2/\Q)\isom S_3$. Let $J=k_2\cap k_3k_5'$ defining a normal subgroup $H\subseteq \Gal(k_3/\Q)\times \Gal(k_5'/\Q)$. Suppose first $\Gal(k_2/\Q)\isom S_3$. If $H$ is of index 2, then $J\subset \Q(\zeta_3, \zeta_5+\zeta_5^{-1})$ and $k_2\cap k_\ell'=\Q$ imply $J=\Q(\zeta_3)$ or $\Q(\zeta_3(\zeta_5+\zeta_5^{-1}))$. If $k_2$ is not contained in $k_3$, then $H\cap \Gal(k_3/\Q)\times \{1\}$ is index 2 inside $\Gal(k_3/\Q)$. Then $H\cap \{1\}\times \Gal(k_5'/\Q)$ must be a normal subgroup of index dividing 3. Since $\Gal(k_5'/\Q)$ has no normal sbugroup of index 3, it must be index 1. Hence $H$ has index 2, so $J=\Q(\zeta_3)$. If $\Gal(k_2/\Q)\isom\Z/3\Z$, then $\Gal(k_3/\Q)\times \Gal(k_5'/\Q)$ has no normal subgroup of index 3. Hence $J=\Q$.

Let $\sigma_3\in NF_3$ whose image in $S_3$ has order 3. Then $\sigma_2\in NF_2$ is compatible with $\sigma_3$. Let $\sigma_5\in\Gal(k_5'/\Q)$ correspond to the class of the matrix $A$ in $\PGL_2(\F_5)$. By the two paragraphs above, there exists $\sigma'\in \Gal(k_2k_3k_5'/\Q)$ which lifts $\sigma_2,\sigma_3,\sigma_5$. If $K\subset k_2k_3k_5'$, then the only quadratic extension fixed by $\sigma'$ is $\Q(\zeta_3)$. Thus, there exists $\sigma\in\Gal(k_2k_3k_5'K/\Q)$ lifting $\sigma$ and acting nontrivially on $K$. Let $\tau_1\in\Gal(k_2k_3k_5K/\Q)$ be any lift of $\sigma$. We claim that $\tau_1|_{k_5}\in NF_5$. Indeed, observe that $\rho_5(\tau_1|_{k_5})$ can be any one of the following matrices
$$\begin{bmatrix} 0&3\\3&4\end{bmatrix},\begin{bmatrix} 0&1\\1&3\end{bmatrix},\begin{bmatrix} 0&4\\4&2\end{bmatrix},\begin{bmatrix} 0&2\\2&1\end{bmatrix}$$
all of which lie in $NF_5$.

Finally $k_2k_3k_5K\cap k_{n/30}\subset\Q(\zeta_{30})$ by Lemma \ref{lem:gl2quot}. Then by Proposition \ref{prop:surj}, there exists $\tau\in\Gal(k_nK/\Q)$ such that $\tau$ restricts to $\tau_1$, $\tau|_{k_n}\in NF_n$, and $K^\tau=\Q$. 
\qedhere
\end{subcase}
\end{proof}

\begin{cor}\label{ellred}
Let $E/\Q$ be a elliptic curve and $a\in\Z$ satisfying the following conditions.

\begin{enumerate}
    \item For any isogenous $E'$, $E'(\Q)$ is torsion free.
    \item The Galois representation $\rho_\ell\colon\Gal(\Qbar/\Q)\to\GL_2(\Z/\ell\Z)$ on the torsion points on $E$ is surjective for all $\ell>2$ except for possibly one prime.
    \item If $\rho_\ell$ is not surjective for some $\ell>2$, then $\Q(E[2])\cap \Q(E[\ell])=\Q$.
    \item $\sqrt{a}\notin \Q(E[2]), \Q(\sqrt{-3})$.
\end{enumerate}

Then for any integer $n$, there are infinitely many primes $p$ such that $E_p(\F_p)[n]=0$ and $p$ is inert in $\Q(\sqrt{a})$.
\end{cor}

\begin{proof}
This is a direct consequence of Theorem \ref{thm:torcomposite} applied to $K=\Q(\sqrt{a})$.
\end{proof}

\begin{remark}\label{rmk:conditions}
It is clear that condition (1) is necessary for Corollary \ref{ellred} to hold. Condition (4) is also necessary as shown in Proposition \ref{prop:twotorsion} and Proposition \ref{prop:threetorsion}. Conditions (2) and (3) ensure we do not deal with too many intersections of division fields that are not surjective in the proof. We expect these may not be necessary.
\end{remark}

\bibliographystyle{alpha}
\bibliography{conicbundle}

\end{document}